\documentclass{article}

\usepackage[T1]{fontenc}
\usepackage{tgtermes}
\usepackage{mathptmx}  
\usepackage{amsthm,amsmath,amssymb}
\usepackage{mathrsfs}
\usepackage{enumerate}

\usepackage{graphicx}
\usepackage{amssymb}
\usepackage{amsmath}
\newtheorem{theorem}{Theorem}
\newtheorem{lemma}{Lemma}
\newtheorem{proposition}{Proposition}
\newtheorem{definition}{Definition}
\newtheorem{corollary}{Corollary}

\title{Lighthouse Principle for Diffusion\\ in Social Networks}

\author{Sanaz Azimipour \and Pavel Naumov}

\begin{document}

\maketitle

\begin{abstract}
The article investigates influence relation between two sets of agents in a social network. It proposes a logical system that captures propositional properties of this relation valid in all threshold models of social networks with the same topological structure. The logical system consists of Armstrong axioms for functional dependence and an additional Lighthouse axiom. The main results are soundness, completeness, and decidability theorems for this logical system. 
\end{abstract}

\section{Introduction}

In this article we study influence in social networks. When a new product is introduced to the market, it is usually first adopted by a few users that are called ``early adopters". These users might adopt the product because they are fans of the company introducing the product, as a result of the marketing campaign conducted by the company, or because they have a genuine need for this type of product. Once the early adopters start using the product, they put peer pressure on their friends and acquaintances in the social network, who might eventually follow them in adopting the product. The friends of the early adopters might eventually influence their own friends and so on, until the product is potentially adopted by a significant part of the network. 

A similar phenomenon could be observed with diffusion of certain behaviours, like smoking, adoption of new words and technical innovations, and propagation of beliefs.  

There are two most widely used models that formally capture diffusion process in social networks. One of them is the {\em cascading} model~\cite{snk08kiies,kkt05alp}. This model distinguishes active and inactive vertices of the network. Once a vertex $v$ becomes active, it gets a single chance to activate each neighbour $u$ with a given probability $p_{v,u}$. This process continues until no more activations can happen. 

% http://link.springer.com/chapter/10.1007/11523468_91

% http://link.springer.com/chapter/10.1007/978-3-540-85567-5_9

%http://en.wikipedia.org/wiki/Information_cascade

In this article we focus on the second model, called {\em threshold} model~\cite{v96sn,macy91asr,kkt03sigkdd,am14fi}, originally introduced by Granovetter~\cite{g78ajs} and Schelling~\cite{s78}. In this model each agent has a non-negative threshold value representing the agent's resistance to adoption of a given product. If the pressure from those peers of the agent who already adopted the product reaches the threshold value, then the agent also adopts the product. We assume that each of the other agents has a non-negative, but possibly zero, influence on the given agent. The peer pressure on an agent to adopt a product is the sum of influences on the agent of all agents who have already adopted the product. It is assumed in this model that, once the product is adopted, the agent keeps using the product and putting pressure on her peers indefinitely.

\begin{figure}[ht]
\begin{center}
%\vspace{3mm}
\scalebox{0.5}{\includegraphics{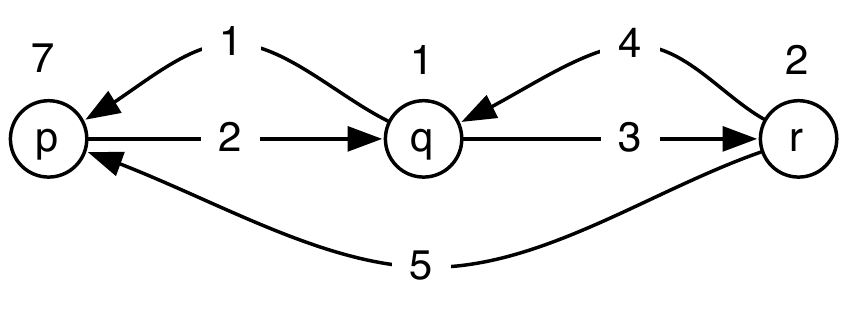}}
\vspace{0mm}
%\footnotesize
\caption{Social Network $N_1$
}\label{intro-example-1-fig}
\vspace{0cm}
\end{center}
%\vspace{-2mm}
\end{figure}

Consider, for example, social network $N_1$ depicted in Figure~\ref{intro-example-1-fig}. This network consists of three agents: $p$, $q$, and $r$ that have threshold values $7$, $1$, and $2$ respectively. The influence of one agent on another is shown in this figure by the label on the directed edge connecting the two agents. For instance, the influence of agent $r$ on agent $p$ is 5. If an agent has zero influence on another agent, then appropriate directed edge is not shown at all. Thus, influence of agent $p$ on agent $r$ is zero.

Suppose that a marketing company gives agent $p$ a free sample of the product and the agent starts using it. Since agent $p$ has influence $2$ on agent $q$ and threshold value of agent $q$ is only $1$, she will eventually also adopt the product. In turn, adoption of the product by agent $q$ will eventually lead to adoption of the product by agent $r$ because threshold value of agent $r$ is only $2$ and the influence of agent $q$ on agent $r$ is $3$. Thus, adoption of the product by agent $p$ eventually leads to adoption of this product by agent $r$. We denote this fact by $N_1\vDash p\rhd r$.

In this article we study relation $A\rhd B$ between group of agents $A$ and $B$ that could be informally described\footnote{We formally specify this relation in Definition~\ref{sat}.} as ``if all agents in set $A$ are given free samples of the product and they all start using it, then all agents in set $B$ will eventually adopt the product". For example, for the discussed above social network $N_1$, we have $N_1\vDash\{p\}\rhd\{q,r\}$, which we usually write as just $N_1\vDash p\rhd q,r$. 

At the same time, if a free sample of the product is given to agent $r$, then agent $q$ will eventually adopt it because she has threshold value $1$ and the influence of agent $r$ on her is $4$. Once agent $q$ adopts the product, however, the product diffusion stops and the product will never be adopted by agent $p$ because her threshold value is $7$ and the total peer pressure from agents $q$ and $r$ on $p$ will be only $1+5=6$. Therefore, for example, $N_1\vDash \neg (r\rhd p)$. 

The properties of relation $A\rhd B$ that we have discussed so far were specific to social network $N_1$. Let us now consider social network $N_2$ depicted in Figure~\ref{intro-example-2-fig}.
\begin{figure}[ht]
\begin{center}
%\vspace{3mm}
\scalebox{0.5}{\includegraphics{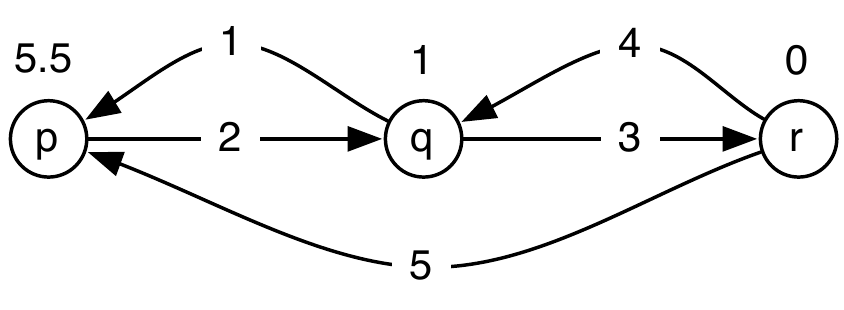}}
\vspace{0mm}
%\footnotesize
\caption{Social Network $N_2$
}\label{intro-example-2-fig}
\vspace{0cm}
\end{center}
%\vspace{-2mm}
\end{figure}
If a free sample of the product is given in network $N_2$ to agent $r$ and she starts using it, then, like it was for the network $N_1$, agent $q$ will eventually adopt the product because her threshold value is only $1$ and influence of agent $r$ on agent $q$ is $4$. Unlike network $N_1$, however, the product diffusion does not stop at this point because now total peer pressure of agents $q$ and $r$ on agent $p$ is still $1+5=6$, but the threshold value of agent $p$ in this network is only $5.5$. Thus, agent $p$ eventually will adopt the product. In other words, $N_2\vDash r\rhd p$. 

An interesting property of network $N_2$ is that agent $r$ has threshold value $0$. Thus, she will eventually adopt the product even if no free product samples are given to any of the agents: $N_2\vDash \varnothing\rhd r$. 

Note that social networks $N_1$ and $N_2$ are different only by the threshold values that the agents have. The agents in both networks have the same influence on each other. We will express this by saying that social networks $N_1$ and $N_2$ have the same {\em sociogram}. This common sociogram $S_1$ for networks $N_1$ and $N_2$ is depicted in Figure~\ref{example-1-fig}.

\begin{figure}[ht]
\begin{center}
%\vspace{3mm}
\scalebox{0.5}{\includegraphics{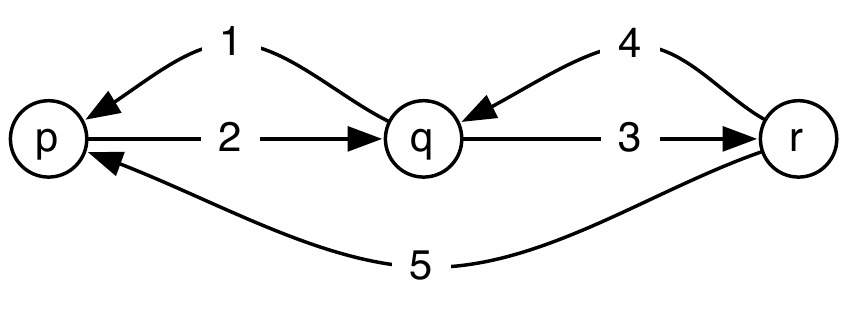}}
%\footnotesize
\caption{Sociogram $S_1$}\label{example-1-fig}
\vspace{0cm}
\end{center}
%\vspace{-2mm}
\end{figure}

To some degree, the threshold values characterize the relation that exists between the product and the individual agents and the sociogram describes the influence relation between the agents. The term sociogram has been first introduced by psychosociologist Jacob Levy Moreno~\cite{m34}.
% in 1934 and further developed in 1953 edition of the same book~\cite[p.141]{m53}.
The sociograms, as defined in this article, are directed weighted graphs. The original Moreno's sociograms were neither directed nor weighted. We  briefly discuss the unweighted sociograms in the conclusion. 

In this article we study not the individual properties of specific social networks, but the common properties of all social networks with the same sociogram. We write $S\vDash\phi$ if property $\phi$ is true for all social networks with sociogram $S$.  For example, as we show in Proposition~\ref{example-1},
\begin{equation}\label{intro claim}
S_1\vDash p\rhd r \to q\rhd r.
\end{equation}
In other words, under any assignment of threshold values on sociogram $S_1$, if giving a free sample of the product to agent $p$ will eventually lead to agent $r$ adopting the product, then giving a free sample of the product to agent $q$ would have the same effect.

The main result of this article is a complete axiomatization of propositional properties of relation $A\rhd B$ for any given sociogram. Such axiomatization consists of three axioms common to all sociograms and a sociogram-specific fourth  axiom. The first three axioms are
\begin{enumerate}
\item Reflexivity: $A\rhd B$ if $B\subseteq A$,
\item Transitivity: $A\rhd B \to (B\rhd C \to A\rhd C)$,
\item Augmentation: $A\rhd B\to (A,C\rhd B,C)$, 
\end{enumerate}
where $A,B$ denotes the union of sets $A$ and $B$. These axioms were originally proposed by Armstrong~\cite{a74} to describe functional dependence relation in database theory.  They became known in database literature as Armstrong's axioms \cite[p.~81]{guw09}. V{\"a}{\"a}n{\"a}nen proposed a first order version of these principles~\cite{v07} and their generalization for reasoning about approximate dependency~\cite{v14arxiv}. Beeri, Fagin, and Howard~\cite{bfh77} suggested a variation of Armstrong's axioms that describes properties of multi-valued dependence. Naumov and Nicholls~\cite{nn14jpl} proposed another variation of these axioms that describes rationally functional dependence.

The sociogram-dependent fourth axiom captures the fact that in every group of agents in which at least one agent eventually adopts the product there is always an agent (or a subgroup of agents) who adopts the product first. In marketing such agents are sometimes called {\em lighthouse customers}. In any given group of agents, the distinctive property of lighthouse customers is that they adopt the product without any peer pressure coming from other agents in this group. The lighthouse customers adopt the product as a result of the peer pressure from the outside of the group. Our fourth axiom postulates existence of lighthouse customers in any group of agents in which at least one agent eventually will adopt the product. Thus, we call this postulate {\em Lighthouse axiom}.

\begin{figure}[ht]
\begin{center}
%\vspace{3mm}
\scalebox{0.5}{\includegraphics{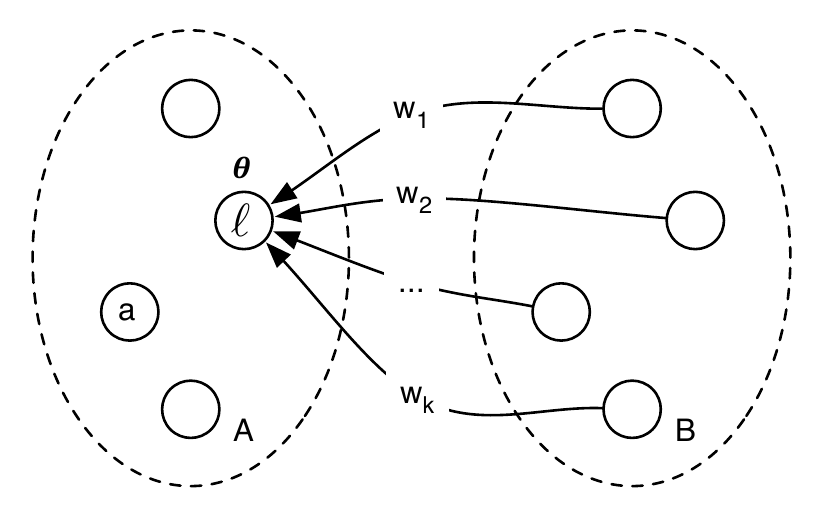}}
\vspace{0mm}
%\footnotesize
\caption{Lighthouse Axiom 
}\label{lighthouse-fig}
\vspace{0cm}
\end{center}
%\vspace{-2mm}
\end{figure}

One possible way to state Lighthouse axiom is to say that if all agents in network $N$  are partitioned into  disjoint sets $A$ and $B$, see Figure~\ref{lighthouse-fig}, and there is an agent $a\in A$ such that $N\vDash B\rhd a$, then there must exist a ``lighthouse" agent $\ell\in A$ such that the total peer pressure of all agents in set $B$ on agent $\ell$ is no less than the threshold value of agent $\ell$:
$$\theta\le w_1+w_2+\dots + w_k.$$
Unfortunately, when stated this way, Lighthouse axiom refers to threshold value $\theta$ of agent $\ell$. Thus, in this form, it is a property of the social network, rather than the corresponding sociogram. 

It turns out, however, that there is a way to re-word the axiom so that it does not refer to threshold values. Namely, let us assume that for every agent $a\in A$ we choose a set of agents $C_a\subseteq A\cup B$ such that peer pressure of set $C_a$ on agent $a$ is no less than peer pressure of set $B$ on agent $a$. The new form of Lighthouse axiom states that, under the above condition, if $N\vDash B\rhd a$, then there exists a ``lighthouse" agent $\ell\in A$ such that $N\vDash C_\ell\rhd \ell$. The main result of this article is the completeness theorem for logical system consisting of this form of Lighthouse axiom and the three Armstrong axioms.

Several logical frameworks for reasoning about diffusion in social networks have been studied before. Seligman, Liu, and Girard~\cite{slg11lia} proposed Facebook Logic for capturing properties of epistemic social networks in modal language, but did not give any axiomatization for this logic. They further developed this approach in papers~\cite{slg13tark,lsg14synthese}. In particular, they introduced dynamic friendship relations. Christoff and Hansen~\cite{ch15jal} simplified Seligman, Liu, and Girard setting and gave a complete axiomatization of the logical system for this new setting.
Christoff and Rendsvig proposed Minimal Threshold Influence Logic~\cite{cr14elisiem} that uses modal language to capture dynamic of diffusion in a threshold model and gave complete axiomatization of this logic. The languages of the described above systems are significantly different from ours and, as a result, neither of these systems  contains principles similar to our Lighthouse axiom.  

Diffusion in social networks is a special case of information flow on graphs. Logical systems for reasoning about various types of graph information flow has been studied before. Lighthouse axiom has certain resemblance with Gateway axiom for functional dependence on hypergraphs of secrets~\cite{mn11clima}, Contiguity axiom~\cite{hn13sr} for graphical games, and Shield Wall axiom for fault tolerance in belief formation networks~\cite{hn12jelia}.

This article is organized as following. In Section~\ref{syntax section} we introduce formal syntax and semantics of our logical system. Section~\ref{axioms section} list the four axioms of the system. In Section~\ref{examples section}, we give several examples of formal proofs in our system. In Section~\ref{star section} we show some auxiliary results that are used later. Section~\ref{soundness section} and Section~\ref{completeness section} prove soundness and completeness theorems respectively. Section~\ref{conclusion section} concludes with a discussion of logical properties of unweighted sociograms.

\section{Syntax and Semantics}\label{syntax section}

In this section we formally define social network, sociogram, and influence relation. 

\begin{definition}\label{Phi}
For any finite set $\mathcal{A}$, let $\Phi(\mathcal{A})$ be the minimal set of formulas such that
\begin{enumerate}
\item $\bot\in\Phi(\mathcal{A)}$,
\item $A\rhd B\in \Phi(\mathcal{A})$, for each subsets $A,B\subseteq \mathcal{A}$,
\item $\phi\to\psi\in \Phi(\mathcal{A})$ for each $\phi,\psi\in\Phi(\mathcal{A})$.
\end{enumerate}
\end{definition}
We assume that disjunction $\vee$ is defined through implication $\to$ and false constant $\bot$ in the standard way.

\begin{definition}\label{social network}
A social network is triple $(\mathcal{A},w,\theta)$, where
\begin{enumerate}
\item $\mathcal{A}$ is an arbitrary finite set (of agents),
\item $w$ is a function that maps $\mathcal{A}^2$ into non-negative real numbers. Value $w(a,b)$ represents influence of agent $a$ on agent $b$.
\item $\theta$ is a function that maps $\mathcal{A}$ into non-negative real numbers. Value $\theta(a)$ represents threshold value of agent $a\in\mathcal{A}$.
\end{enumerate}
\end{definition}

\begin{definition}\label{socigram}
A sociogram is pair $(\mathcal{A},w)$, where set $\mathcal{A}$ and function $w$ satisfy the first two conditions of Definition~\ref{social network}.
\end{definition}
We say that social network $(\mathcal{A},w,\theta)$ is based on sociogram $(\mathcal{A},w)$.
We now proceed to define peer pressure on an agent by a group of agents in a given sociogram. 
\begin{definition}\label{norm def}
For any sociogram $(\mathcal{A},w)$ and any subset of agents $A\subseteq\mathcal{A}$, let
$\|A\|_b=\sum_{a\in A}w(a,b)$.
\end{definition}

In the introduction we said that if, at some moment in time, an agent  experience peer pressure higher than her threshold value, then at some point in the future she will adopt the product. For the sake of simplicity, in our formal model we assume that time is discrete and that if at moment $k$ an agent experiences sufficient peer pressure, then she adopts the product at moment $k+1$. Although this assumption, generally speaking, affects the ``time dynamics" of product diffusion, it does not affect the final outcome of diffusion. Thus, this assumption, while simplifying the formal setting, does not change the properties of influence relation $A\rhd B$. Given this assumption, if free samples of the product are given to all agents in set $A$ at moment $0$, then by $A^k$ we mean the set of all agents who will adopt the product by moment $k$. The formal definition of $A^k$ is below.  

\begin{definition}\label{Ak}
For any $A\subseteq \mathcal{A}$ and any $k\in \mathbb N$, let subset $A^k\subseteq \mathcal{A}$ be defined recursively as follows:
\begin{enumerate}
\item $A^0=A$,
\item $A^{k+1}=A^k \cup \{x \in \mathcal{A}\;|\; \|A^k\|_x\ge \theta(x)\}$.
\end{enumerate}
\end{definition}

\begin{corollary}\label{add powers}
$(A^n)^k=A^{n+k}$.
\end{corollary}

If free samples of the product are given to all agents in set $A$, then by $A^*$ we mean the set of all agents who will eventually adopt the product. The formal definition of $A^*$ is below.

\begin{definition}\label{A*}
$$A^*=\bigcup_{k\ge 0}A^k.$$
\end{definition}

The next definition specifies the formal semantics of our logical system. In particular, item 2 in this definition specifies the formal meaning of the influence relation.

\begin{definition}\label{sat}
For any social network 
$N=(\mathcal{A},w,\theta)$ 
and any 
$\phi\in \Phi(\mathcal{A})$, let satisfiability relation
$N\vDash\phi$ be defined as follows
\begin{enumerate}
\item $N\nvDash\bot$,
\item $N\vDash A\rhd B$ if $B\subseteq A^*$,
\item $N\vDash \psi\to\chi$ if $N\nvDash \psi$ or $N\vDash \chi$.
\end{enumerate}
\end{definition}

\section{Axioms} \label{axioms section}

Our logical system for an arbitrary sociogram $S=(\mathcal{A},w)$ consists of propositional tautologies in language $\Phi(\mathcal{A})$ and the following additional axioms:

\begin{enumerate}
\item Reflexivity: $A\rhd B$ if $B\subseteq A$,
\item Transitivity: $A\rhd B \to (B\rhd C \to A\rhd C)$,
\item Augmentation: $A\rhd B\to (A,C\rhd B,C)$, 
\item Lighthouse: if $A\sqcup B$ is a partition of the set of all agents $\mathcal{A}$ and $\{C_a\}_{a\in A}$ is a family of sets of agents such that $\|B\|_{a}\le \|C_a\|_{a}$ for each $a\in A$,
then
$$
\bigvee_{a\in A} B\rhd a  \to \bigvee_{a\in A} C_a\rhd a.
$$
\end{enumerate}

We write $\vdash_S\phi$ if formula $\phi$ can be derived in our system using Modus Ponens inference rule. We sometimes write just $\vdash\phi$ if the value of subscript $S$ is clear from the context. We also write $X \vdash_S\phi$ if formula $\phi$ could be derived in our system extended by the set of additional axioms $X$.

\section{Examples}\label{examples section}

In this section we give three examples of formal proofs in our logical system.
Soundness of this system is shown in Section~\ref{soundness section}. We start by proving statement~(\ref{intro claim}) from the introduction.

\begin{proposition}\label{example-1}
$\vdash_{S_1} p\rhd r\to q\rhd r$, where $S_1$ is the sociogram depicted in Figure~\ref{example-1-fig}.
\end{proposition}

\begin{figure}[ht]
\begin{center}
%\vspace{3mm}
\scalebox{.5}{\includegraphics{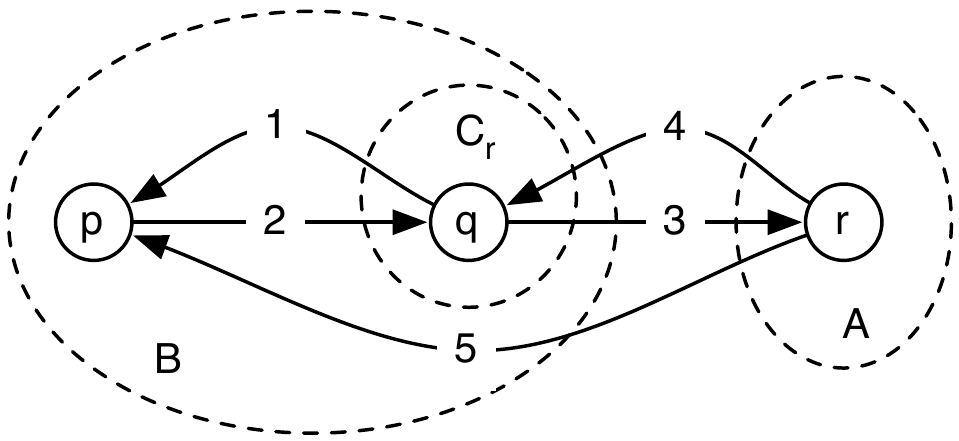}}
%\footnotesize
\caption{Towards Proof of Proposition~\ref{example-1}}\label{example-1-proof-fig}
\vspace{0cm}
\end{center}
%\vspace{-2mm}
\end{figure}

\begin{proof}
Let $A=\{r\}$, $B=\{p,q\}$, and $C_r=\{q\}$, see Figure~\ref{example-1-proof-fig}. Note that
$$\|B\|_r=w(p,r)+w(q,r)=0+3= 3 = w(q,r)=\|C_r\|_r.$$
Hence, by Lighthouse axiom,
\begin{equation}\label{example-1-eq}
\vdash p,q\rhd r \to q\rhd r.
\end{equation}
At the same time, by Transitivity axiom, 
$$
\vdash p,q\rhd p\to (p\rhd r\to p,q\rhd r).
$$
By Reflexivity axiom, $\vdash p,q\rhd p$. Thus, by Modus Ponens inference rule,
$$
\vdash p\rhd r\to p,q\rhd r.
$$
Therefore, $\vdash p\rhd r\to q\rhd r$ using statement (\ref{example-1-eq}) and propositional logic reasoning.
\end{proof}

\begin{figure}[ht]
\begin{center}
%\vspace{3mm}
\scalebox{.5}{\includegraphics{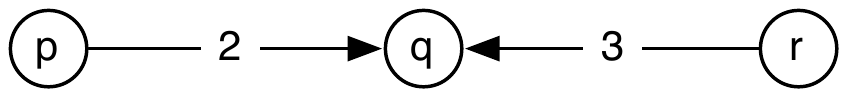}}
\vspace{5mm}
%\footnotesize
\caption{Sociogram $S_2$
}\label{example-2-fig}
\vspace{0cm}
\end{center}
%\vspace{-2mm}
\end{figure}

Let us now consider sociogram $S_2$ depicted on Figure~\ref{example-2-fig}. Since on this sociogram agent $r$ has higher influence on agent $q$ than agent $p$ has, one might expect the following statement to be true for all social networks over  sociogram $S_2$: 
\begin{equation}\label{pq to rq}
p\rhd q\to r\rhd q.
\end{equation}
Surprisingly, this is false. Namely, this statement is false for the social network depicted in Figure~\ref{example-2-note-fig}.
\begin{figure}[ht]
\begin{center}
%\vspace{3mm}
\scalebox{.5}{\includegraphics{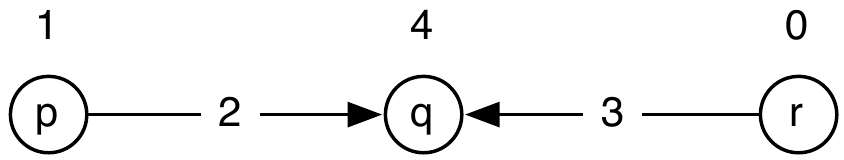}}
\vspace{5mm}
%\footnotesize
\caption{Social Network
}\label{example-2-note-fig}
\vspace{0cm}
\end{center}
%\vspace{-2mm}
\end{figure}
This happens because agent $r$ in this social network has threshold value $0$. In other words, agent $r$ is an ``early adopter" who does not need any external peer pressure in order to buy the product. As a result, see Figure~\ref{example-2-note-star-fig}, we have $\{p\}^1=\{p,r\}$. Once agent $r$ adopts the product, the total peer pressure on agent $q$ becomes $2+3=5$ and she will adopt the product as well. On the other hand, if the free sample is given to agent $r$, then neither agent $p$ nor agent $q$ ever adopt the product.

\begin{figure}[ht]
\begin{center}
%\vspace{3mm}
\scalebox{.5}{\includegraphics{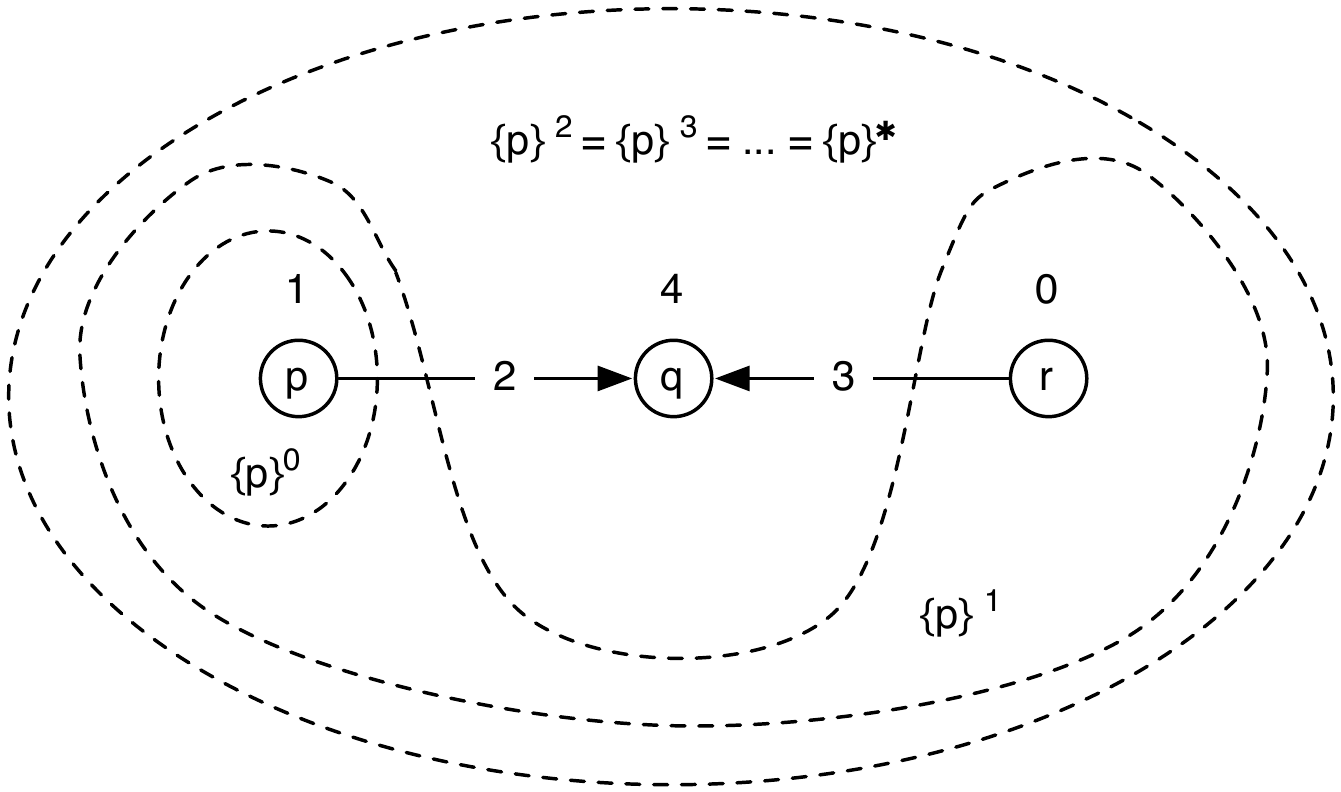}}
\vspace{5mm}
%\footnotesize
\caption{Social Network
}\label{example-2-note-star-fig}
\vspace{0cm}
\end{center}
%\vspace{-2mm}
\end{figure}
Although statement~(\ref{pq to rq}) holds not for all social networks over sociogram $S_2$, in the next proposition we show that a slightly modified version of this statement does hold for all such networks.

\begin{proposition}\label{example-2}
$\vdash_{S_2} p\rhd q\to (r\rhd q \vee \varnothing\rhd r)$, where $S_2$ is the sociogram depicted in Figure~\ref{example-2-fig}.
\end{proposition}
\begin{proof}
Let $A=\{q,r\}$, $B=\{p\}$, $C_q=\{r\}$, and $C_r=\varnothing$, see Figure~\ref{example-2-proof-fig}.
\begin{figure}[ht]
\begin{center}
\vspace{3mm}
\scalebox{.5}{\includegraphics{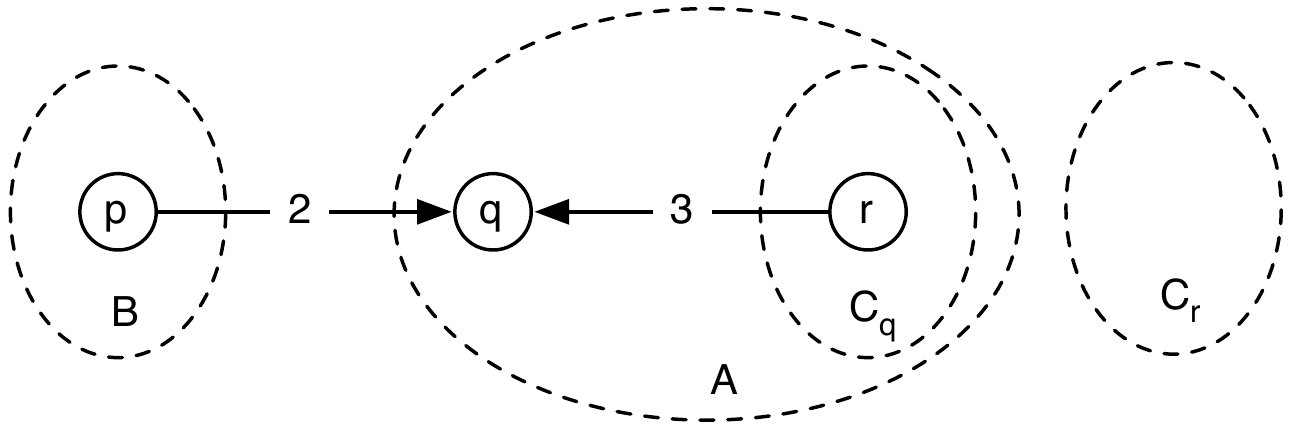}}
%\footnotesize
\caption{Towards Proof of Proposition~\ref{example-2}}\label{example-2-proof-fig}
\vspace{0cm}
\end{center}
\vspace{-2mm}
\end{figure}
Note that
$$
\|B\|_q=w(p,q)=2<3=w(r,q)=\|C_q\|_q
$$
and
$$
\|B\|_r=w(p,r)=0=\|\varnothing\|_r=\|C_r\|_r.
$$
Thus, by Lighthouse axiom,
$$
\vdash p\rhd q\vee p\rhd r\to r\rhd q \vee \varnothing \rhd r.
$$
Therefore,
$
\vdash p\rhd r\to r\rhd q \vee \varnothing \rhd r.
$
\end{proof}

\begin{figure}[ht]
\begin{center}
%\vspace{3mm}
\scalebox{.5}{\includegraphics{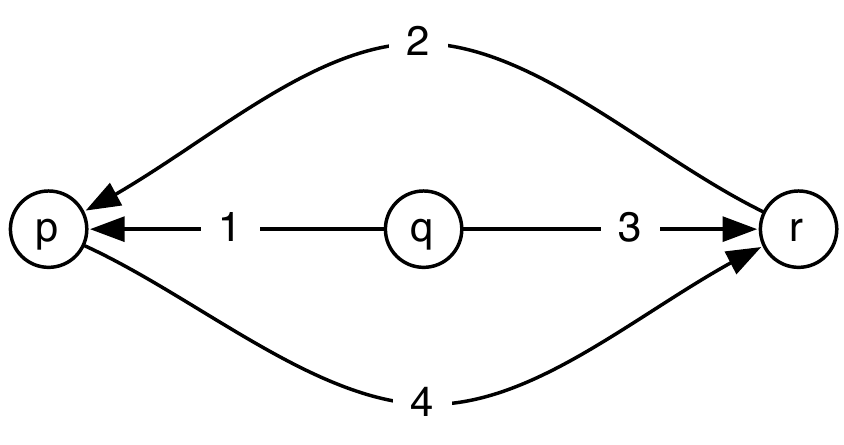}}
\vspace{5mm}
%\footnotesize
\caption{Sociogram $S_3$
}\label{example-3-fig}
\vspace{0cm}
\end{center}
%\vspace{-2mm}
\end{figure}

\begin{proposition}\label{example-3}
$\vdash_{S_3} q\rhd p \vee  q\rhd r\to p\rhd r \vee r\rhd p$, where $S_3$ is the sociogram depicted in Figure~\ref{example-3-fig}.
\end{proposition}

\begin{figure}[ht]
\begin{center}
%\vspace{3mm}
\scalebox{.5}{\includegraphics{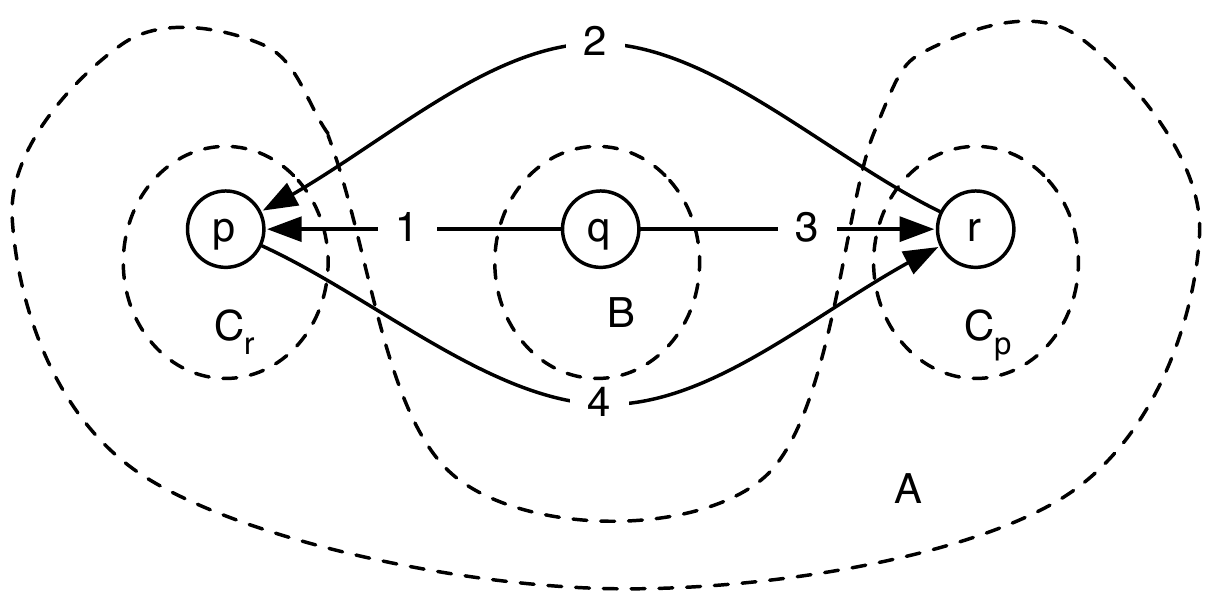}}
%\footnotesize
\caption{Towards Proof of Proposition~\ref{example-3}}\label{example-3-proof-fig}
\vspace{0cm}
\end{center}
%\vspace{-2mm}
\end{figure}

\begin{proof} 
Let $A=\{p,r\}$, $B=\{q\}$, $C_p=\{r\}$, and $C_r=\{p\}$,
see Figure~\ref{example-3-proof-fig}. Note that
$$
\|B\|_p=w(q,p)=1<2=w(r,p)=\|C_p\|_p
$$
and
$$
\|B\|_r=w(q,r)=3<4=w(p,r)=\|C_r\|_r.
$$
Therefore, by Lighthouse axiom,
$$
\vdash q\rhd p\vee q\rhd r\to p\rhd r \vee r \rhd p.
$$
\end{proof}

\section{Properties of Star Closure}\label{star section}

In this section we prove several technical properties of $A^*$ that are used later in the proofs of soundness and completeness.

\begin{lemma}\label{A1 imp Ak}
If $A^1=A$, then $A^k=A$ for each $k\ge 0$.
\end{lemma}
\begin{proof}
We prove this lemma by induction on $k$. If $k=0$, then $A^0=A$ by Definition~\ref{Ak}. If $k>0$, then by Corollary~\ref{add powers}, assumption $A^1=A$, and the induction hypothesis,
$$
A^{k}=(A^1)^{k-1}=A^{k-1}=A.
$$
\end{proof}

\begin{lemma}\label{A*=Ak}
$A^*=A^k$ for some $k\ge 0$.
\end{lemma}
\begin{proof}
The statement of the lemma follows from the assumption in Definition~\ref{social network} that set $\mathcal{A}$ is finite.
\end{proof}

\begin{lemma}\label{theta>}
If $x\notin A^*$, then $\theta(x)>\|A^*\|_x$, for each subset $A\subseteq \mathcal{A}$ and each agent $x\in \mathcal{A}$.
\end{lemma}
\begin{proof}
By Lemma~\ref{A*=Ak}, there is $k\ge 0$ such that $A^*=A^k$. Suppose that $\|A^*\|_x\ge \theta(x)$. Thus, $\|A^k\|_x\ge \theta(x)$. Hence, $x\in A^{k+1}$, by Definition~\ref{Ak}. Thus, $x\in A^*$ by Definition~\ref{A*}, which is a contradiction to the assumption of the lemma.
\end{proof}

\begin{lemma}\label{A subseteq A*}
$A\subseteq A^*$.
\end{lemma}
\begin{proof}
By Definition~\ref{Ak} and Definition~\ref{A*},
$$
A=A^0\subseteq \bigcup_{k\ge 0}A^k=A^*.
$$
\end{proof}

\begin{lemma}\label{A**}
$(A^*)^*\subseteq A^*$.
\end{lemma}
\begin{proof}
By Lemma~\ref{A*=Ak}, there are $n,k\ge 0$ such that $A^*=A^n$ and $(A^*)^*=(A^*)^k$. Thus, by Corollary~\ref{add powers} and Definition~\ref{A*},
$$
(A^*)^*=(A^*)^k=(A^n)^k=A^{n+k}\subseteq \bigcup_{m\ge 0} A^m = A^*.
$$
\end{proof}

\begin{lemma}\label{subset Ak}
If $A\subseteq B$, then $A^k\subseteq B^k$, for each $k\ge 0$.
\end{lemma}
\begin{proof}
We prove the statement of the lemma by induction on $k$. If $k=0$, then $A^0=A\subseteq B=B^0$ by Definition~\ref{Ak}.

Suppose that $A^k\subseteq B^k$. Let $x\in A^{k+1}$. It suffices to show that $x\in B^{k+1}$. Indeed, by Definition~\ref{Ak}, assumption $x\in A^{k+1}$ implies that either $x\in A^{k}$ or $\|A^{k}\|_x\ge \theta(x)$. In the first case, by the induction hypothesis, $x\in A^{k}\subseteq B^{k}$. Thus, $x\in B^{k}$. Therefore, $x\in B^{k+1}$ by Definition~\ref{Ak}.

In the second case, by Definition~\ref{norm def} and assumption $A^k\subseteq B^k$,
$$
\|B^k\|_x=\sum_{b\in B^k} w(b,x)\ge \sum_{a\in A^k} w(a,x) = \|A^k\|_x\ge \theta(x).
$$
Therefore, $x\in B^{k+1}$ by Definition~\ref{Ak}.
\end{proof}

\begin{corollary}\label{star mono}
If $A\subseteq B$, then $A^*\subseteq B^*$.
\end{corollary}

\begin{lemma}\label{star union}
$A^*\cup B^*\subseteq (A\cup B)^*$.
\end{lemma}
\begin{proof}
Note that $A\subseteq A\cup B$ and $B\subseteq A\cup B$. Thus, $A^*\subseteq (A\cup B)^*$ and $B^*\subseteq (A\cup B)^*$ by Corollary~\ref{star mono}. Therefore, $A^*\cup B^*\subseteq (A\cup B)^*$.
\end{proof}

\section{Soundness}\label{soundness section}

In this section we prove the soundness of our logical system with respect to the semantics given in Definition~\ref{sat}. The soundness of propositional tautologies and Modus Ponens inference rule is straightforward. Below we show the soundness of each of the remaining four axioms as separate lemmas. In the lemmas that follow we assume that $S=(\mathcal{A},w,\theta)$ is a social network and $A$, $B$, and $C$ are subsets of $\mathcal{A}$.

\begin{lemma}
If $B\subseteq A$, then $S\vDash A\rhd B$.
\end{lemma}
\begin{proof}
By Lemma~\ref{A subseteq A*}, $A\subseteq A^*$. Thus, $B\subseteq A^*$ by the assumption of the lemma. Therefore, $S\vDash A\rhd B$, by Definition~\ref{sat}.
\end{proof}

\begin{lemma}  
If $ S \vDash A \rhd B $ and $  S \vDash B \rhd C $, then $ S \vDash A \rhd C $. 
\end{lemma}
\begin{proof}
By Definition~\ref{sat}, assumption $S \vDash A \rhd B$ implies that $B\subseteq A^*$. Hence, $B^*\subseteq (A^*)^*$ by Corollary~\ref{star mono}. Thus, $B^*\subseteq A^*$ by Lemma~\ref{A**}. At the same time, $C\subseteq B^*$ by assumption $S \vDash B \rhd C$ and Definition~\ref{sat}. Thus, $C\subseteq A^*$. Therefore, $S \vDash A \rhd C$ by Definition~\ref{sat}.
\end{proof}

\begin{lemma}  
If $ S \vDash A \rhd B $, then $ S \vDash A,C \rhd B,C $. 
\end{lemma}
\begin{proof} 
Suppose that $ S \vDash A \rhd B $. Thus, $B\subseteq A^*$ by Definition~\ref{sat}. Note that $C\subseteq C^*$ by Lemma~\ref{A subseteq A*}. Thus, 
$$
B\cup C \subseteq A^* \cup C^* \subseteq (A\cup C)^*,
$$
by Lemma~\ref{star union}. Therefore, $ S \vDash A,C \rhd B,C $, by Definition~\ref{sat}.
\end{proof}

\begin{lemma}
If $S\vDash B\rhd a_0$ for some $a_0\in A$, then there is $\ell\in A$ such that $S\vDash C_{\ell}\rhd \ell$, where $A\sqcup B$ is a partition of the set of all agents $\mathcal{A}$ and $\{C_a\}_{a\in A}$ is a family of sets of agents such that $\|B\|_{a}\le \|C_a\|_{a}$ for each $a\in A$.
\end{lemma}
\begin{proof}
Note that assumption $S\vDash B\rhd a_0$ by Definition~\ref{sat} implies that $a_0\in B^*$. On the other hand, assumption $a_0\in A$ implies that $a_0\notin B$ because $A\sqcup B$ is a partition of set $\mathcal{A}$. Thus, $B^*\neq B$. Hence, by Definition~\ref{A*}, there must exist $k$ such that $B^k\neq B$. Then, $B^1\neq B$ by Lemma~\ref{A1 imp Ak}. Thus, there must exist $\ell \in B^1\setminus B$. Hence, $\|B\|_\ell\ge \theta(\ell)$ by Definition~\ref{Ak}. Then, by the assumption of the lemma, $\|C_\ell\|_\ell\ge \|B\|_\ell\ge \theta(\ell)$. Thus, $\ell \in C^1_\ell$, by Definition~\ref{Ak}. Hence, $\ell\in C^*_\ell$ by Definition~\ref{A*}. Therefore, $S\vDash C_\ell\rhd \ell$ by Definition~\ref{sat}. Finally, note that $\ell\in A$ because $\ell\in B^1\setminus B$ and $A\sqcup B$ is a partition of the set $\mathcal{A}$.
\end{proof}
This concludes the proof of the soundness of our logical system.

\section{Completeness}\label{completeness section}

In this section we proof the completeness of our logical system with respect to the semantics given in Definition~\ref{sat}. This result is formally stated as Theorem~\ref{completeness theorem} in the end of this section. The proof of completeness theorem consists in  the construction of a ``canonical" social network. We start, however, we a few technical lemmas and definitions.

\subsection{Preliminaries}

Let us first prove a useful property of real numbers.
\begin{lemma}\label{x and y}
If $\epsilon>0$ is a real number and $x$ and $y$ are any real numbers such that either $x=y$ or $|x-y|>\epsilon$. Then, $x+\epsilon > y$ implies $x\ge y$.
\end{lemma}
\begin{proof}
Suppose $y>x$. Hence, $x\ne y$. Thus, $|x-y|>\epsilon$, by the assumption of the lemma. Then, $y-x>\epsilon$, because $y>x$. Therefore, $x+\epsilon<y$.
\end{proof}

We now assume a fixed sociogram $(\mathcal{A},w)$ and a fixed maximal consistent subset $X$ of $\Phi(\mathcal{A})$.
\begin{definition}\label{hat def}
$\widehat{A}=\{a\in \mathcal{A}\;|\; X\vdash A\rhd a\}$ for each subset $A\subseteq \mathcal{A}$.
\end{definition}
Choose $\epsilon$ to be any positive real number such that $\epsilon < \|A\|_a -\|B\|_a$
for each agent $a\in \mathcal{A}$ and each subsets $A,B\subseteq \mathcal{A}$, such that $\|A\|_a > \|B\|_a$. This could be achieved because set $\mathcal{A}$ is finite.

\begin{lemma}\label{A B lemma}
For any subsets $A,B\subseteq \mathcal{A}$ and any agent $a\in\mathcal{A}$ if $\|A\|_a+\epsilon>\|B\|_a$, then $\|A\|_a\ge\|B\|_a$.
\end{lemma}
\begin{proof}
By the choice of $\epsilon$, we have either $\|A\|_a=\|B\|_a$ or $|(\|A\|_a-\|B\|_a)|>\epsilon$. Thus, $\|A\|_a\ge\|B\|_a$ by Lemma~\ref{x and y}.
\end{proof}

\begin{lemma}\label{A sub Ahat}
$A\subseteq\widehat{A}$ for each subset $A\subseteq\mathcal{A}$.
\end{lemma}
\begin{proof}
Suppose that $a\in A$. Thus, $\vdash A\rhd a$ by Reflexivity axiom. Therefore, $a\in \widehat{A}$ by Definition~\ref{hat def}. 
\end{proof}

\begin{lemma}\label{ArhdAhat}
$X\vdash A\rhd \widehat{A}$, for each subset $A\subseteq \mathcal{A}$.
\end{lemma}
\begin{proof}
Let $\widehat{A}=\{a_1,\dots,a_n\}$. By the definition of $\widehat{A}$,  $X\vdash A\rhd a_i$, for any $i\le n$. We prove, by induction on $k$, that $X\vdash A\rhd a_1,\dots,a_k$ for each $0\le k\le n$. 

\noindent {\em Base Case}: $X\vdash A\rhd \varnothing$ by Reflexivity axiom.

\noindent {\em Induction Step}: Assume that $X\vdash A\rhd a_1,\dots,a_k$. By Augmentation axiom, 
\begin{equation}\label{eq0}
X\vdash A, a_{k+1}\rhd a_1,\dots,a_k,a_{k+1}.
\end{equation}
Recall that $X\vdash A\rhd a_{k+1}$. Again by Augmentation axiom, $X\vdash A\rhd A, a_{k+1}$.
Hence, $X\vdash A \rhd a_1,\dots,a_k,a_{k+1}$, by (\ref{eq0}) and Transitivity axiom.
\end{proof}

\subsection{Canonical Social Network}

Next, based on the sociogram $(\mathcal{A},w)$ and the maximal consistent set $X$, we define ``canonical" social network $N_X=(\mathcal{A},w,\theta)$. We then proceed to prove the core properties of this network.

\begin{definition}\label{theta def}
$$
\theta(a)=
\begin{cases}
0, & \mbox{ if $X\vdash \varnothing\rhd a$,}\\
\max_{a\notin \widehat{B}}\|\widehat{B}\|_a+\epsilon, & \mbox{ otherwise.}
\end{cases}
$$
\end{definition}
The maximum in the above definition is taken over all subsets $B$ of $\mathcal{A}$ such that $\widehat{B}$ does not contain agent $a$.
\begin{lemma}
Function $\theta(a)$ is well-defined for each $a\in\mathcal{A}$.
\end{lemma}
\begin{proof}
We need to show that if $X\nvdash \varnothing\rhd a$, then there is at least one subset $B\subseteq \mathcal{A}$ such that $a\notin\widehat{B}$. It suffices to show that $a\notin\widehat{\varnothing}$, which is true due to assumption $X\nvdash \varnothing\rhd a$ and Definition~\ref{hat def}.
\end{proof}

\begin{lemma}\label{Cx exists}
For any subset $B\subseteq\mathcal{A}$,
if $a\in \mathcal{A}\setminus B^*$, then there is $C\subseteq \mathcal{A}$ such that $a\notin \widehat{C}$ and $\theta(a)=\|\widehat{C}\|_a+\epsilon$.
\end{lemma}
\begin{proof}
If $\theta(a)=0$, then, $a\in B^1$ due to Definition~\ref{Ak}. Thus, $a\in B^*$ by Definition~\ref{A*}, which is a contradiction to the assumption $a\in \mathcal{A}\setminus B^*$.

Suppose now that $\theta(a)>0$, thus, by Definition~\ref{theta def}, there is at least one $C\subseteq \mathcal{A}$ such that $a\notin \widehat{C}$ and $\theta(a)=\|\widehat{C}\|_a+\epsilon$.
\end{proof}

\begin{lemma}\label{theta lemma}
If $B\subseteq \mathcal{A}$ and  $a\in\mathcal{A}\setminus\widehat{B}$, then
$\theta(a)>\|\widehat{B}\|_a$.
\end{lemma}
\begin{proof}
{\em Case I:} $X\vdash\varnothing\rhd a$. Note that $X\vdash B\rhd \varnothing$ by Reflexivity axiom. Thus, $X\vdash B\rhd a$ by Transitivity axiom. Hence, $a\in \widehat{B}$ by Definition~\ref{hat def}, which is a contradiction to the assumption of the lemma.

{\em Case II:} if $X\nvdash\varnothing\rhd a$, then $\theta(a)>\|\widehat{B}\|_a$ by Definition~\ref{theta def}.
\end{proof}

\begin{lemma}\label{hat k}
$(\widehat{B})^k=\widehat{B}$ for each $B\subseteq\mathcal{A}$ and each $k\ge 0$.
\end{lemma}
\begin{proof}
We prove this statement by induction on $k$. If $k=0$, then $(\widehat{B})^k=\widehat{B}$, by Definition~\ref{Ak}.
Note next that by Definition~\ref{Ak}, the induction hypothesis, and Lemma~\ref{theta lemma},
\begin{eqnarray*}
(\widehat{B})^{k+1}&=&(\widehat{B})^k\cup \{a\in\mathcal{A}\;|\; \|(\widehat{B})^k\|_a \ge \theta(a)\}\\
&=& \widehat{B}\cup \{a\in\mathcal{A} \;|\; \|\widehat{B}\|_a \ge \theta(a)\}\\
&=& \widehat{B}\cup \{a\in\mathcal{A}\setminus \widehat{B} \;|\; \|\widehat{B}\|_a \ge \theta(a)\}
= \widehat{B}\cup \varnothing = \widehat{B}.
\end{eqnarray*}
\end{proof}

\begin{lemma}\label{hat star}
$(\widehat{B})^*=\widehat{B}$ for each $B\subseteq\mathcal{A}$.
\end{lemma}
\begin{proof}
By Definition~\ref{A*} and Lemma~\ref{hat k},
$$
(\widehat{B})^* = \bigcup_{k\ge 0}(\widehat{B})^k=\bigcup_{k\ge 0}\widehat{B}=\widehat{B}.
$$
\end{proof}

\begin{lemma}\label{star lemma}
For each $B\subseteq\mathcal{A}$,
if  $a\in B^*$, then $X\vdash B\rhd a$.
\end{lemma}
\begin{proof}
Suppose $a\in B^*$. By Lemma~\ref{A sub Ahat}, $B\subseteq \widehat{B}$. Then, $B^*\subseteq (\widehat{B})^*$ Corollary~\ref{star mono}. Thus, $a\in (\widehat{B})^*$. Hence, $a\in \widehat{B}$ by Lemma~\ref{hat star}. Therefore, $X\vdash B\rhd a$ by Definition~\ref{hat def}.
\end{proof}

\begin{lemma}\label{start lemma 2}
For each $B\subseteq\mathcal{A}$ and each $a\in\mathcal{A}$,
if $X\vdash B\rhd a$, then $a\in B^*$.
\end{lemma}
\begin{proof}
By Lemma~\ref{theta>},  $\theta(x)>\|B^*\|_x$ for each $x\in \mathcal{A}\setminus B^*$.
At the same time, by Lemma~\ref{Cx exists}, for each $x\in \mathcal{A}\setminus B^*$ there is $C_x$ such that $x\notin \widehat{C_x}$ and $\theta(x)=\|\widehat{C_x}\|_x+\epsilon$.
Hence, $\|\widehat{C_x}\|_x+\epsilon > \|B^*\|_x$ for each $x\in \mathcal{A}\setminus B^*$.
Thus, by Lemma~\ref{A B lemma}, $\|\widehat{C_x}\|_x\ge  \|B^*\|_x$ for each $x\in \mathcal{A}\setminus B^*$.

Consider partition $(\mathcal{A}\setminus B^*)\sqcup B^*$ of $\mathcal{A}$.
By Lighthouse axiom,
\begin{equation}\label{lighthouse application}
\vdash \bigvee_{x\in \mathcal{A}\setminus B^*}B^*\rhd x\to \bigvee_{x\in \mathcal{A}\setminus B^*}\widehat{C_x}\rhd x.
\end{equation}
Suppose that $a\notin B^*$, 
Lemma~\ref{A subseteq A*} and Reflexivity axiom imply that $\vdash B^*\rhd B$. Thus, by assumption $X\vdash B\rhd a$ and Transitivity axiom, $X\vdash B^*\rhd a$. Hence, statement (\ref{lighthouse application}) implies that
$$
X \vdash \bigvee_{x\in \mathcal{A}\setminus B^*}\widehat{C_x}\rhd x.
$$

Then, due to the maximality of set $X$, there must exist $x_0\in \mathcal{A}\setminus B^*$ such that $X\vdash \widehat{C_{x_0}}\rhd x_0$. Thus, $X\vdash C_{x_0}\rhd x_0$, due to Lemma~\ref{ArhdAhat} and Transitivity axiom:
$$
\vdash C_{x_0}\rhd \widehat{C_{x_0}}\to (\widehat{C_{x_0}}\rhd x_0\to C_{x_0}\rhd x_0).
$$
Hence, $x_0\in \widehat{C_{x_0}}$ by Definition~\ref{hat def}, which is a contradiction with the choice of $C_x$.
\end{proof}

\begin{lemma}\label{main induction}
$N_X\vDash \phi$ if and only if $\phi\in X$, 
for each formula $\phi\in\Phi(\mathcal{A})$.
\end{lemma}
\begin{proof}
We prove this lemma by induction on structural complexity of formula $\phi$. Cases when formula $\phi$ is $\bot$ or has form $\psi_1\to\psi_2$ follow in the standard way from Definition~\ref{sat} and the assumptions of maximality and consistency of set $X$. Suppose that $\phi$ has form $A\rhd B$.  

$(\Rightarrow):$
Suppose that $N_X \vDash A \rhd B$.
Then $ B \subseteq A^* $ by Definition~\ref{sat}. 
Hence, $ b \in A^* $ for each $ b \in B$.
Thus, $ X\vdash A \rhd b$ for each $b \in B$ by
Lemma~\ref{star lemma}.
Hence, $b \in \widehat{A} $ for each $ b \in B$ by  Definition~\ref{hat def}.
In other words, $B\subseteq \widehat{A}$. 
Thus, by Reflexivity axiom, $ \vdash \widehat{A} \rhd B$. On the other hand, $X \vdash A \rhd \widehat{A} $ by Lemma~\ref{ArhdAhat}. 
Therefore,  $ X\vdash A \rhd B$ by Transitivity axiom.

$(\Leftarrow):$
Assume $ X \vdash A \rhd  B$. 
By Reflexivity axiom, $ \vdash B \rhd b$  for every $b \in B $. 
Hence, $ X  \vdash A \rhd b$ for each $ b\in B$ by Transitivity axiom.
Thus,  $b \in A^*$ for each $b \in B$, by Lemma~\ref{start lemma 2}.
In other words, $B \subseteq A^*$.
Therefore, $ N_X \vDash A \rhd B$ by Definition~\ref{sat}.
\end{proof}

\subsection{Main Result}

We are now ready to state and prove the completeness theorem for our logical system with respect to the semantics given in Definition~\ref{sat}.

\begin{theorem}\label{completeness theorem}
For any sociogram $(\mathcal{A},w)$ and any formula $\phi\in\Phi(\mathcal{A})$, if $N\vDash\phi$ for each social network $N$ based on sociogram $(\mathcal{A},w)$, then $\vdash \phi$.
\end{theorem}
\begin{proof}
Suppose that $\nvdash \phi$. Let $X$ be a maximal consistent subset of $\Phi(\mathcal{A})$ such that $\phi\notin X$. By Lemma~\ref{main induction}, $N_X\nvDash \phi$.
\end{proof}

\section{Decidability}

In this section we discuss decidability of our logical system for any fixed sociogram $(\mathcal{A},w)$. Note that we allow arbitrary real numbers as subscripts in formula $A\rhd_c B$. Thus, the set of all formulas $\Phi(\mathcal{A})$ is uncountable and its elements can not be used as inputs of a Turing machine. In order to avoid this issue, in this section we modify Definition~\ref{Phi}, Definition~\ref{social network}, and Definition~\ref{socigram} by assuming that only rational numbers could be used as subscripts in our atomic formulas $A\rhd_c B$, as influence values, and as threshold values. It is easy to see that the above proof of completeness is still valid. From this change point of view, the only non-trivial place is the choice of $\epsilon$ for the given sociogram $(\mathcal{A},w)$ that we have made right after Definition~\ref{hat def}. Note, however, that the required $\epsilon$ could always be choose to be a rational number because 0 is a limit point of the set of positive rational numbers.

\begin{theorem}
For any given sociogram $S=(\mathcal{A},w)$, set $\{\phi\in\Phi(\mathcal{A}) \;|\; \vdash_S\phi\}$ is decidable.
\end{theorem}
\begin{proof}
According to Theorem~\ref{completeness theorem}, $\vdash_S\phi$ if and only if formula $\phi$ is true for each social network $(\mathcal{A},w,\theta)$ based on sociogram $S$. This, of course, does not imply the decidability because there are infinitely many social networks based on sociogram $S$. However, it turns out that the proof of Theorem~\ref{completeness theorem} that we gave above actually shows a stronger result: $\vdash_S\phi$ if and only if formula $\phi$ is true for each social network from a specific finite class $C(S)$ of networks based on sociogram $S$. 

Once existence of such {\em finite} class of social networks $C(S)$ is establish, we should be able to claim the decidability result because one can always verify if a formula $\phi$ is true for each out of {\em finitely} many given networks.

We are now ready to describe finite class of social networks $C(S)$. The social network over sociogram $S$ is completely defined by specifying threshold function $\theta$. In the proof of Theorem~\ref{completeness theorem}, this is done in Definition~\ref{theta def}. This definition depends on $\epsilon$ and maximal consistent set of formulas $X$. Note however that the choice of $\epsilon$ does not depend on $X$ and could be made based on sociogram $S$ alone. Once $\epsilon$ is fixed, the set of all values of function $\theta$, as specified in Definition~\ref{theta def}, belongs to {\em finite} set
$$\{0\}\cup \{\|A\|_a+\epsilon\;|\; a\in\mathcal{A}, A\subseteq \mathcal{A}\}.$$
The set of all social networks over sociogram $S$ whose threshold functions use only values from the above set is the desired finite class of social networks $C(S)$.
\end{proof}

\section{Conclusion}\label{conclusion section}

In this article we have studied properties of influence common to all social networks with the same weighted sociogram. We described all such properties in the propositional language by introducing a logical system for reasoning about these properties and proving soundness and completeness of this system. We have established that the logical system is decidable if its syntax and semantics are restricted to rational numbers.

A natural extension of this work is to consider common influence properties of social network with the same {\em unweighted} sociogram, in which presence of a directed edge between two agents means that one agent might have non-zero influence on the other agent. Absence of a directed edge means that the influence of one agent on the other is zero.

Although an unweighted sociogram is a simpler structure than weighted sociogram, there are some very non-trivial properties of influence relation common to all social networks with the same unweighted sociogram. 

\begin{figure}[ht]
\begin{center}
%\vspace{3mm}
\scalebox{0.5}{\includegraphics{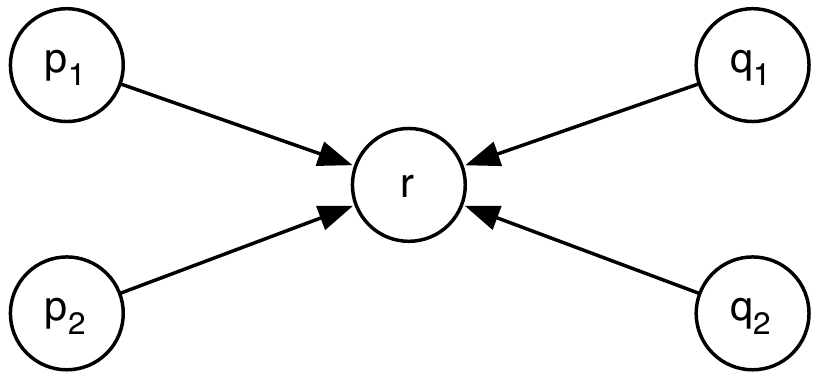}}
\vspace{0mm}
%\footnotesize
\caption{Unweighted Sociogram $U_1$
}\label{conclusion-example-1-fig}
\vspace{0cm}
\end{center}
%\vspace{-2mm}
\end{figure}

Consider, for example, unweighted sociogram $U_1$ depicted in Figure~\ref{conclusion-example-1-fig}. Let $N=(\mathcal{A},w,\theta)$ be a social network based on $U_1$. Furthermore, assume that in social network $N$ (i) neither of the agents $p_1,p_2,q_1,q_2$ is an early adopter, (ii) $N\vDash p_1,p_2\rhd r$, and (iii) $N\vDash q_1,q_2\rhd r$. Thus, $w(p_1,r)+w(p_2,r)\ge \theta(r)$ and $w(q_1,r)+w(q_2,r)\ge \theta(r)$. The first inequality implies that at least one out of $w(p_1,r)$ and $w(p_2,r)$ is greater or equal than $\theta(r)/2$. In other words, there is $i\in\{1,2\}$ such that $w(p_i,r)\ge \theta/2$. Similarly, the second inequality implies that there is $j\in\{1,2\}$ such that $w(q_j,r)\ge \theta/2$. Thus,
$$
\|\{p_i,q_j\}\|_r = w(p_i,r) + w(q_j,) \ge \theta/2 + \theta/2 =\theta.
$$
Hence, $r\in \{p_i,q_j\}^1\subseteq \{p_i,q_j\}^*$. Then, $N\vDash p_i,q_j\rhd r$. So, we have shown that for any social network $N$ based on unweighted sociogram $U_1$ and satisfying the conditions (i), (ii), (iii), there are $i,j\in\{1,2\}$ such that $N\vDash p_i,q_j\rhd r$. This could be formally stated as
\begin{eqnarray*}
&U_2\vDash& p_1 , p_2 \rhd r \wedge q_1 , q_2 \rhd r\\
&& \to \bigvee_{i=1}^2 \bigvee_{j=1}^2 p_i, q_j \rhd r \vee \bigvee_{x\in\{p_1,p_2,q_1,q_2\}} \varnothing \rhd x,
\end{eqnarray*}
where disjunction $\bigvee_{x\in\{p_1,p_2,q_1,q_2\}} \varnothing \rhd x$ captures the statement that one of agents $p_1,p_2,q_1,q_2$ is an early adopter. The above principle is just an example of a non-trivial property of diffusion  common to all social networks with the same unweighted sociogram. This example can be stated in a more general form as
\begin{eqnarray*}
&U_2\vDash & \bigwedge_{i=1}^np_{i1},p_{i2},\dots,p_{in}\rhd q\\ 
&& \to \bigvee_{j_1=1}^n\bigvee_{j_2=1}^n\dots \bigvee_{j_n=1}^n p_{1 j_1},p_{2 j_2},\dots,p_{n j_n}\rhd q\\
&&  \;\;\;\;\; \vee \bigvee_{i=1}^n\bigvee_{j=1}^n \varnothing \rhd p_{ij},
\end{eqnarray*}
where $U_2$ is unweighted sociogram depicted in Figure~\ref{conclusion-example-2-fig}.
\begin{figure}[ht]
\begin{center}
\vspace{3mm}
\scalebox{0.5}{\includegraphics{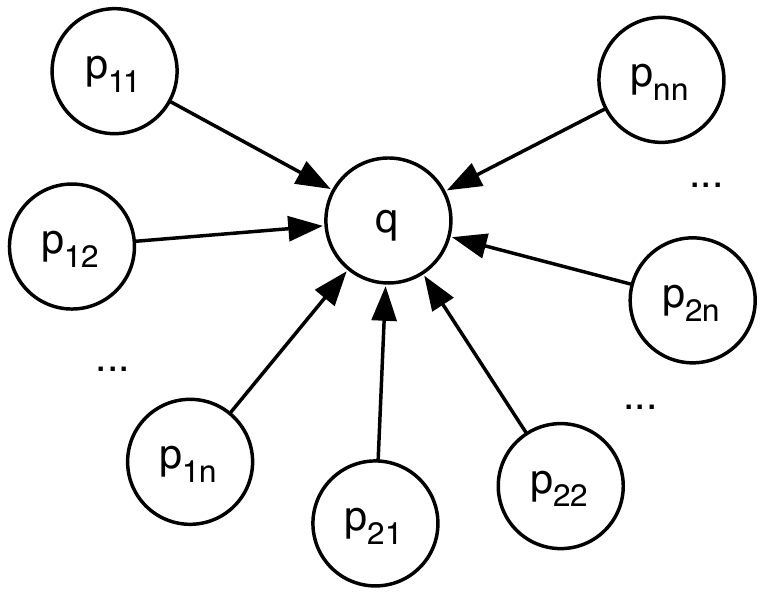}}
\vspace{0mm}
%\footnotesize
\caption{Unweighted Sociogram $U_2$
}\label{conclusion-example-2-fig}
\vspace{0cm}
\end{center}
\vspace{-2mm}
\end{figure}
Complete axiomatization of properties of influence common to all social networks with a given unweighted sociogram remains an open problem.

\bibliography{sp}

\end{document}